\documentclass[12pt]{amsart}
\usepackage{mathrsfs}

\usepackage{amssymb,amsfonts,amsthm,amsmath}
\usepackage{epsfig}
\usepackage[utf8]{inputenc}
\usepackage{mathrsfs}
\usepackage{lscape}
\usepackage{amssymb,amsmath,graphicx,color,textcomp, amsthm,bbm,bbold, enumerate,booktabs}
\usepackage{chngcntr}
\usepackage{apptools}
\AtAppendix{\counterwithin{theorem}{section}}
\definecolor{Red}{cmyk}{0,1,1,0}

\definecolor{verde}{cmyk}{1,0,1,0}

\definecolor{loka}{cmyk}{.5,0,1,.5}
\definecolor{azul}{cmyk}{1,1,0,0}

\numberwithin{equation}{section}

\newcommand{\be}{\begin{equation}}
\newcommand{\ee}{\end{equation}}

\newtheorem{definition}{Definition}

\newtheorem{corolario}{Corollary}

\newtheorem{teorema}{Theorem}
\newtheorem{remark}{Remark}

\begin{document}
\title{Stability of fractional functional differential equations}
\author{J. Vanterler da C. Sousa}
\address{ Department of Applied Mathematics, Institute of Mathematics,
 Statistics and Scientific Computation, University of Campinas --
UNICAMP, rua S\'ergio Buarque de Holanda 651,
13083--859, Campinas SP, Brazil\newline
e-mail: {\itshape \texttt{vanterlermatematico@hotmail.com, capelas@ime.unicamp.br }}}

\address{$^1$ Departamento de Matem\'aticas, Universidad de la Serena, Benavente 980, La Serena, Chile \newline
e-mail: {\itshape \texttt{fabio.granrod@gmail.com}}}
\author{E. Capelas de Oliveira}
\author{F. G. Rodrigues$^1$}

\begin{abstract} In this paper, we present a study on the Ulam-Hyers and Ulam-Hyers-Rassias stabilities of the solution of the fractional functional differential equation using the Banach fixed point theorem.

\vskip.5cm
\noindent
\emph{Keywords}: $\psi$-Hilfer fractional derivative, Ulam-Hyers stability, Ulam-Hyers-Rassias stability, fractional functional differential equations, Banach fixed point theorem.
\newline 
MSC 2010 subject classifications. 26A33, 34A08, 34K37, 34K20.
\end{abstract}
\maketitle

\section{Introduction}

An exchange of questions and answers between Ulam and Hyers, the research on
stability of solutions of functional differential equations was started
several years ago \cite{hyers,ulam}. More precisely, Ulam raised the following
question: Let $H_{1}$ and $H_{2}$ be a group and a metric group endowed with
the metric $d(\cdot,\cdot)$, respectively. Given $\varepsilon>0$, does there
exists a $\delta>0$ such that if the function $f:H_{1}\rightarrow H_{2}$
satisfies the following inequality $d(f(x,y),f(x)f(y))<\delta$, $\forall
x,y\in H_{1}$, then there exists a homeomorphism $F:H_{1}\rightarrow H_{2}$
with $d(f(x),F(x))<\varepsilon$, $\forall x\in H_{1}$? And so Hyers, presents
his answer, in the case where $H_{1}$ and $H_{2}$ are Banach spaces
\cite{ulam}. Since then, a significant number of researchers have devoted
themselves to developing their research which address stability and many
important works have been published not only on functional differential
equations, but also other types of equations \cite{S1,S2,S3,S4,S5}.

On the other hand, with the expansion of the fractional calculus and the
number of researchers investigating more and more problems involving the
stability of solutions of fractional functional differential equations
specially in Banach spaces, this field started to gain more attention
\cite{ista1,ista2,ista3,ista4}. In addition, not only stability has been the
subject of study, but investigating the existence and uniqueness, as well as
the controllability of solutions of fractional differential equations, has
called, and still call, a lot of attention over the years
\cite{exis1,exis2,exis3,exis4,cont1,cont2,cont3}.

In 2012 Zhao et al. \cite{exis5}, investigated the existence of positive
solutions of the fractional functional differential equation introduced by
means of the Caputo fractional derivative and using the Krasnosel'skii fixed
point theorem. In this paper, the results obtained on the existence of
positive solutions for the fractional functional differential equation improve
and generalize the existing results. There are numerous works on the existence
and uniqueness of fractional functional differential equations, both locally
and globally in the Hilbert, Banach and Fr\'echet spaces. For better reading
we suggest the works \cite{exis6,exis7,exis8,exis9}.

In the middle of 2017, Abbas et al. \cite{ista4}, investigated the existence
of Ulam-Hyers and Ulam-Hyers-Rassias stabilities of the random solution of the
fractional functional differential equation of the Hilfer and Hilfer-Hadamard
type by means of fixed point theorems. Abbas et al. \cite{ista1}, also
investigated the Ulam stability of functional partial differential equations
through Picard's operator theory and provided some examples. Further work on
stability of fractional functional differential equations and even functional
integral equation can be found in the following works
\cite{ist1,ist2,ist3,ist4}. The stability study is broad and there are other
types of stability in which we will not discuss in this paper, but in the
paper of Stamova and Stamov \cite{exis10}, they perform a system stability
analysis of fractional functional differential equations using the Lyapunov
method and the principle of comparison.

Since the theory about the Ulam-Hyers stability of functional differential
equations is in wide growth, and the number of papers is yet small, one of the
objectives for the realization of this paper is to provide an investigation of
the fractional differential equation Eq.(\ref{eq1}), in order to be a good
research material in this matter.

Consider the delay fractional differential equation of the form
\begin{equation}
\label{eq1}^{H}\mathbb{D}_{t_{0}+}^{\alpha,\beta;\psi}y\left(  t\right)
=F\left(  t,y\left(  t\right)  ,y\left(  t-a\right)  \right)
\end{equation}
where $^{H}\mathbb{D}_{t_{0}+}^{\alpha,\beta;\psi}\left(  \cdot\right) $ is
the $\psi$-Hilfer fractional derivative with $0<\alpha\leq1$, $0\leq\beta
\leq1,$ $F:\mathbb{R}^{3}\rightarrow\mathbb{R} $ is a bounded and continuous
function, $a>0$ is a real constant and $t>a$.

Motivated by the works \cite{mot1,mot2,mot3}, in this paper, we have as main
purpose to investigate the Ulam-Hyers and Ulam-Hyers-Rassias stabilities of
the fractional functional differential equation Eq.(\ref{eq1}) by means of
Banach fixed point theorem.

This paper is organized as follows: in Section 2, we present as preliminaries
the continuous functions and the weighted function space, in order to
introduce the $\psi$- Riemann-Liouville fractional integral and the $\psi
$-Hilfer fractional derivative. In this sense, we present the concepts of
Ulam-Hyers and Ulam-Hyers-Rassias stabilities, as well as Banach fixed point
theorem, which is fundamental for obtaining the main results. In Section 3, we
present the first result of this paper, the Ulam-Hyers-Rassias stability by
means of Banach fixed point theorem. In Section 4, again by means of Banach
fixed point theorem, we present the second result of this paper, the
Ulam-Hyers stability. Concluding remarks close the paper.

\section{Preliminaries}

In this section we present some important concepts that will be useful to
write our mains results. First, we present the definitions of the $\psi
$-Riemann-Liouville fractional integral and the $\psi$-Hilfer fractional
derivative. In this sense, we present the Ulam-Hyers, Ulam-Hyers-Rassias and
generalized Ulam-Hyers-Rassias stabilities concepts for the $\psi$-Hilfer
fractional derivative. We conclude the section with Banach fixed point
theorem, an important result to obtain the stability of the fractional
functional differential equation.

Let $[a,b]$ $(0<a<b<\infty)$ be a finite interval on the half-axis
$\mathbb{R}^{+}$ and $C[a,b]$, $AC^{n}[a,b]$, $C^{n}[a,b]$ be the spaces of
continuous functions, $n$-times absolutely continuous functions, $n$-times
continuously differentiable functions on $[a,b]$, respectively.

The space of the continuous functions $f$ on $[a,b]$ with the norm is defined
by \cite{ZE1}
\[
\left\Vert f\right\Vert _{C\left[  a,b\right]  }=\underset{t\in\left[  a,b
\right]  }{\max}\left\vert f\left(  t\right)  \right\vert .
\]

On the order hand, we have $n$-times absolutely continuous functions given by
\[
AC^{n}\left[  a,b\right]  =\left\{  f:\left[  a,b\right]  \rightarrow
\mathbb{R} ;\text{ }f^{\left(  n-1\right)  }\in AC\left(  \left[  a,b\right]
\right)  \right\}  .
\]

The weighted space $C_{\gamma,\psi}[a,b]$ of functions $f$ on $(a,b]$ is
defined by \cite{ZE1}
\[
C_{\gamma;\psi}\left[  a,b\right]  =\left\{  f:\left(  a,b\right]
\rightarrow\mathbb{R};\text{ }\left(  \psi\left(  t\right)  -\psi\left(
a\right)  \right)  ^{\gamma}f\left(  t\right)  \in C\left[  a,b\right]
\right\}  ,\text{ }0\leq\gamma<1
\]
with the norm
\[
\left\Vert f\right\Vert _{C_{\gamma;\psi}\left[  a,b\right]  }=\left\Vert
\left(  \psi\left(  t\right)  -\psi\left(  a\right)  \right)  ^{\gamma
}f\left(  t\right)  \right\Vert _{C\left[  a,b\right]  }=\underset{t\in\left[
a,b\right]  }{\max}\left\vert \left(  \psi\left(  t\right)  -\psi\left(
a\right)  \right)  ^{\gamma}f\left(  t\right)  \right\vert .
\]

The weighted space $C_{\gamma;\psi}^{n}\left[  a,b\right] $ of function $f$ on
$(a,b]$ is defined by \cite{ZE1}
\[
C_{\gamma;\psi}^{n}\left[  a,b\right]  =\left\{  f:\left(  a,b\right]
\rightarrow\mathbb{R};\text{ }f\left(  t\right)  \in C^{n-1}\left[
a,b\right]  ;\text{ }f^{\left(  n\right)  }\left(  t\right)  \in
C_{\gamma;\psi}\left[  a,b\right]  \right\}  ,\text{ }0\leq\gamma<1
\]
with the norm
\[
\left\Vert f\right\Vert _{C_{\gamma;\psi}^{n}\left[  a,b\right]
}=\overset{n-1}{\underset{k=0}{\sum}}\left\Vert f^{\left(  k\right)
}\right\Vert _{C\left[  a,b\right]  }+\left\Vert f^{\left(  n\right)
}\right\Vert _{C_{\gamma;\psi}\left[  a,b\right]  }.
\]

For $n=0$, we have, $C_{\gamma,\psi}^{0}\left[  a,b\right]  =C_{\gamma,\psi
}\left[  a,b\right]  $.

The weighted space $C^{\alpha,\beta}_{\gamma,\psi}[a,b]$ is defined by
\[
C_{\gamma;\psi}^{\alpha,\beta}\left[  a,b\right]  =\left\{  f\in
C_{\gamma;\psi}\left[  a,b\right]  ;\text{ }^{H}\mathbb{D}_{a+}^{\alpha
,\beta;\psi}f\in C_{\gamma;\psi}\left[  a,b\right]  \right\}  ,\text{ }%
\gamma=\alpha+\beta\left(  1-\alpha\right)  .
\]

Let $\alpha>0$, $\left[  a,b\right]  $ and $\psi\left(  t\right)  $ be an
increasing and positive monotone function on $\left(  a,b\right]  $, having a
continuous derivative $\psi^{\prime}\left(  t\right)  $ on $\left[  a,b\right]
$. The Riemann-Liouville fractional integral with respect to another function
$\psi$ on $\left[  a,b\right]  $ is defined by \cite{ZE1}
\begin{equation}
\label{eq2}I_{t_{0}+}^{\alpha;\psi}y\left(  t\right)  :=\frac{1}{\Gamma\left(
\alpha\right)  }\int_{t_{0}}^{t}N^{\alpha}_{\psi}(t,s)y\left(  s\right)  ds
\end{equation}
where $\Gamma\left(  \cdot\right)  $ is the gamma function with $0<\alpha
\leq1$ and $N^{\alpha}_{\psi}(t,s):=\psi^{\prime\alpha{-1}}$. The $\psi
$-Riemann-Liouville fractional integral on the left is defined in an analogous way.

On the other hand, left $n-1<\alpha\leq n$ with $n\in\mathbb{N}$, $J=\left[
a,b\right]  $ be an interval such that $-\infty\leq a<b\leq+\infty$ and let
$f,\psi\in C^{n}\left(  \left[  a,b\right] ,\mathbb{R} \right)  $ be two
functions such that $\psi$ is increasing and $\psi^{\prime}\left(  t\right)
\neq0$, for all $t\in J.$ The $\psi$-Hilfer fractional derivative is given by
\cite{ZE1}
\[
^{H}\mathbb{D}_{t_{0}+}^{\alpha,\beta;\psi}y\left(  t\right)  =I_{t_{0}%
+}^{\beta\left(  n-\alpha\right)  ;\psi}\left(  \frac{1}{\psi^{\prime}\left(
t\right)  }\frac{d}{dt}\right)  ^{n}I_{t_{0}+}^{\left(  1-\beta\right)
\left(  n-\alpha\right)  ;\psi}y\left(  t\right)  .
\]
The $\psi$-Hilfer fractional derivative on the left is defined in an analogous way.

Let $X$ be a nonempty set. A function $d:X\times X\rightarrow\left[
0,\infty\right]  $ is called generalized metric on $X$ if, and only if, $d$
satisfies \cite{principal}:

\begin{enumerate}
\item $d\left(  x,y\right)  =0$ if $x=y;$

\item $d\left(  x,y\right)  =d\left(  y,x\right)  $, for all $x,y\in X;$

\item $d\left(  x,z\right)  \leq d\left(  x,y\right)  +d\left(  y,z\right)  ,$
for all $x,y,z\in X.$
\end{enumerate}

For the study of Ulam-Hyers, Ulam-Hyers-Rassias and generalized
Ulam-Hyers-Rassias stabilities, we will adapt such definitions
\cite{mot1,mot2,principal}.

\begin{definition}
For some $\varepsilon\geq0$, $\Phi\in C_{1-\gamma;\psi}\left[  t_{0}-a,t_{0}
\right]  $ and $t_{0},T\in\mathbb{R}$ with $T>t_{0}$, assume that for any
continuous function $f:\left[  t_{0}-a,T\right]  \rightarrow\mathbb{R}$
satisfying
\[
\left\{
\begin{array}
[c]{cll}%
\left\vert ^{H}\mathbb{D}_{t_{0}+}^{\alpha,\beta;\psi}f\left(  t\right)
-F\left(  t,f\left(  t\right)  ,f\left(  t-a\right)  \right)  \right\vert  &
< & \varepsilon,\text{ }t\in\left[  t_{0},T\right] \\
\left\vert f\left(  t\right)  -\Phi\left(  t\right)  \right\vert  & < &
\varepsilon,\text{ }t\in\left[  t_{0}-a,t_{0}\right] ,
\end{array}
\right.
\]
there exists a continuous function $f_{0}:\left[  t_{0}-a,T\right]
\rightarrow\mathbb{R}$ satisfying:
\[
\left\{
\begin{array}
[c]{cll}%
^{H}\mathbb{D}_{t_{0}+}^{\alpha,\beta;\psi}f\left(  t\right)  & = & F\left(
t,f\left(  t\right)  ,f\left(  t-a\right)  \right)  ,\text{ }t\in\left[
t_{0},T\right] \\
f\left(  t\right)  & = & \Phi\left(  t\right)  ,\text{ }t\in\left[
t_{0}-a,t_{0}\right]
\end{array}
\right.
\]
and
\[
\left\vert f\left(  t\right)  -f_{0}\left(  t\right)  \right\vert \leq
K\left(  \varepsilon\right)  ,\text{ }t\in\left[  t_{0}-a,T\right]
\]
where $K\left(  \varepsilon\right)  $ depending of $\varepsilon$ only. Then,
we say that the solution of \textrm{Eq.(\ref{eq1})} is Ulam-Hyers stable.
\end{definition}

\begin{definition}
For some $\varepsilon\geq0$, $\Phi\in C_{1-\gamma;\psi}\left[  t_{0}%
-a,t_{0}\right]  $ and $t_{0},T\in\mathbb{R}$ with $T>t_{0}$, assume that for
any continuous function $f:\left[  t_{0}-a,T\right]  \rightarrow\mathbb{R}$
satisfying
\[
\left\{
\begin{array}
[c]{cll}%
\left\vert ^{H}\mathbb{D}_{t_{0}+}^{\alpha,\beta;\psi}f\left(  t\right)
-F\left(  t,f\left(  t\right)  ,f\left(  t-a\right)  \right)  \right\vert  &
< & \varphi,\text{ }t\in\left[  t_{0},T\right] \\
\left\vert f\left(  t\right)  -\Phi\left(  t\right)  \right\vert  & < &
\varphi,\text{ }t\in\left[  t_{0}-a,t_{0}\right] ,
\end{array}
\right.
\]
there exists a continuous function $f_{0}:[t_{0}-a,T]\rightarrow\mathbb{R}$
satisfying
\[
\left\{
\begin{array}
[c]{cll}%
^{H}\mathbb{D}_{t_{0}+}^{\alpha,\beta;\psi}f_{0}\left(  t\right)  & = &
F\left( t,f_{0}\left(  t\right)  ,f_{0}\left(  t-a\right)  \right)  ,\text{
}t\in\left[ t_{0},T\right] \\
f_{0}\left(  t\right)  & = & \Phi\left(  t\right)  ,\text{ }t\in\left[
t_{0}-a,t_{0}\right]
\end{array}
\right.
\]
and
\[
\left\vert f\left(  t\right)  -f_{0}\left(  t\right)  \right\vert \leq\Phi
_{1},\text{ }t\in\left[  t_{0}-a,T\right]
\]
where $K\left(  \varepsilon\right)  $ depending of $\varepsilon$ only. Then,
we say that the solution of \textrm{Eq.(\ref{eq1})} is the Ulam-Hyers-Rassias stable.
\end{definition}

\begin{definition}
\textrm{Eq.(\ref{eq1})} is generalized Ulam-Hyers stable with respect to
$\phi$ if there exists $c_{\phi}>0$ such that for each solution $y\in
C_{1-\gamma;\psi}^{1}\left(  \left[  t_{0}-a,T\right] , \mathbb{R}\right)  $
to
\[
\left\vert ^{H}\mathbb{D}_{t_{0}+}^{\alpha,\beta;\psi}y\left(  t\right)
-F\left(  t,y\left(  t\right)  ,y\left(  t-a\right)  \right)  \right\vert
\leq\phi\left(  t\right)  ,\text{ }t\in\left[  t_{0}-a,T\right]
\]
there exists a solution $x\in C_{1-\gamma;\psi}^{1}\left(  \left[
t_{0}-a,T\right] , \mathbb{R}\right)  $ of \textrm{Eq.(\ref{eq1})} with
\[
\left\vert y\left(  t\right)  -x\left(  t\right)  \right\vert \leq c_{\phi
}\phi\left(  t\right)  ,\text{ }t\in\left[  t_{0}-a,T\right]  .
\]

\end{definition}

The following is the result of the Banach fixed point theorem, however its
proof will be omitted.

\begin{teorema}
\label{teo1} \textrm{\cite{bana}} Let $\left(  X,d\right)  $ be a generalized
complete metric space. Assume that $\Omega:X\rightarrow X$ is a strictly
contractive operator with the Lipschitz constant $L<1.$ If there exists a
nonnegative integer $k$ such that $d\left(  \Omega^{k+1}x,\Omega^{k}x\right)
<\infty$ for some $x\in X$, then the following are true:

\begin{enumerate}
\item The sequence $\left\{  \Omega^{n}x\right\}  $ converges to a fixed
$x^{\ast}$ of $\Omega$;

\item $x^{\ast}$ is the unique fixed point of $\Omega$ in
\[
X^{\ast}=\left\{  y\in X:d\left(  \Omega^{k}x,y\right)  <\infty\right\}  .
\]

\item If $y\in X^{\ast}$, then
\[
d\left(  y,X^{\ast}\right)  \leq\dfrac{1}{1-L}d\left(  \Omega y,y\right)  .
\]

\end{enumerate}
\end{teorema}


\section{Ulam-Hyers-Rassias stability}

By means of the Banach fixed point theorem, in this section we present the
first result of this paper, the Ulam-Hyers-Rassias stability for the delay
fractional differential equation, Eq.(\ref{eq1}).

\begin{teorema}
\label{teo2} Consider the interval $I=\left[  t_{0}-a,T\right]  $ and suppose
that $F:I\times\mathbb{R}\times\mathbb{R}\rightarrow\mathbb{R}$ is a
continuous function with the following Lipschitz condition:
\[
\left\vert F\left(  t,x,y\right)  -F\left(  t,z,w\right)  \right\vert \leq
L_{1}\left\vert x-z\right\vert +L_{2}\left\vert y-w\right\vert
\]
for all $\left(  t,x,y\right)  ,\left(  t,z,w\right)  \in I\times\mathbb{R}
\times\mathbb{R}$.

Let $\phi:I\rightarrow\left(  0,\infty\right)  $ be a continuous function.
Assume that $\Phi\in C_{1-\gamma;\psi}\left[  t_{0}-a,t_{0}\right]  $,
$K,L_{1}$ and $L_{2}$ are positive constants with
\[
0< K\left(  L_{1}+L_{2}\right) <1
\]
and
\[
\left\vert \dfrac{1}{\Gamma\left(  \alpha\right)  }\int_{t_{0}}^{t} N^{\alpha
}_{\psi}(t,u) \phi\left(  u\right)  du\right\vert \leq K\phi\left( t\right)  ,
\]
for all $t\in I=\left[  t_{0}-a,T\right] $.

Then, if a continuous function $y:I\rightarrow\mathbb{R}$ and $\varphi
:I\rightarrow(0,\infty)$ satisfies
\[
\left\{
\begin{array}
[c]{ccc}%
\left\vert ^{H}\mathbb{D}_{t_{0}+}^{\alpha,\beta;\psi}y\left(  t\right)
-F\left(  t,y\left(  t\right)  ,y\left(  t-a\right)  \right)  \right\vert  &
< & \varphi\left(  t\right)  ,\text{ }t\in\left[  t_{0},T\right] \\
\left\vert y\left(  t\right)  -\Phi\left(  t\right)  \right\vert  & < &
\varphi\left(  t\right)  ,\text{ }t\in\left[  t_{0}-a,t_{0}\right]
\end{array}
\right.
\]
then there exists a unique continuous function $y_{0}:I\rightarrow\mathbb{R}$
such that
\begin{equation}
\label{eq3}\left\{
\begin{array}
[c]{cll}%
^{H}\mathbb{D}_{t_{0}+}^{\alpha,\beta;\psi}y_{0}\left(  t\right)  & = &
F\left(  t,y_{0}\left(  t\right)  ,y_{0}\left(  t-a\right)  \right)  ,\text{
}t\in\left[ t_{0},\text{ }T\right] \\
y_{0}\left(  t\right)  & = & \Phi\left(  t\right)  ,\text{ }t\in\left[
t_{0}-a,t_{0}\right]
\end{array}
\right.
\end{equation}
and
\begin{equation}
\label{eq4}\left\vert y\left(  t\right)  -y_{0}\left(  t\right)  \right\vert
\leq\frac{1}{1-K\left(  L_{1}+L_{2}\right)  }K\phi\left(  t\right)  ,\text{
for all }t\in I.
\end{equation}

\end{teorema}

\begin{proof}
For the proof of this theorem, first consider the set $S$ given by
\begin{equation*}
S=\left\{ \varphi :I\rightarrow \mathbb{R}:\varphi \in C_{1-\gamma;\psi},\text{ }\varphi \left( t\right) =\Phi \left( t\right) ,\text{
if }t\in \left[ t_{0}-a,t_{0}\right] \right\}
\end{equation*}
and the following generalized metric over $S$
\begin{equation}\label{eq5}
d\left( \varphi ,\mu \right) =\inf \left\{ M\in \left[ 0,\infty \right):\left\vert \varphi \left( t\right) -\mu \left( t\right) \right\vert \leq
M\phi \left( t\right) ,\text{ }\forall t\in I\right\}.
\end{equation}
Note that, $\left( S,d\right) $ is generalized complete metric space. Now, we introduce the following function $\Omega :S\rightarrow S$ given by
\begin{equation}\label{eq6}
\left\{
\begin{array}{cll}
\left( \Omega \varphi \right) \left( t\right) & = & \Phi \left( t\right) ,%
\text{ }t\in \left[ t_{0}-a,t_{0}\right] \\
\left( \Omega \varphi \right) \left( t\right) & = &
\begin{array}{l}
\Phi\left( t_{0}\right) \Psi^{\gamma}(t,t_{0})+
\displaystyle\frac{1}{\Gamma \left( \alpha \right) }\displaystyle\int_{t_{0}}^{t}N^{\alpha}_{\psi}(t,s)F\left( u,\varphi \left( u\right) ,\varphi \left( u-a\right) \right) du,\\
\text{ }t\in \left[ t_{0},T\right],
\end{array}%
\end{array}%
\right.
\end{equation}
where $\Psi^{\gamma}(t,t_{0}):=\dfrac{(\psi(t)-\psi(t_{0}))^{1-\gamma}}{\Gamma(\gamma)}$, with $\gamma=\alpha+\beta(1-\alpha)$.
Note that, for $\varphi\in S$, the function $\Omega\varphi$ is continuous. In this way, we can write $\Omega \varphi \in S$.
Let $\varphi ,\mu \in S$ and by Eq.(\ref{eq6}), we have
\begin{eqnarray*}
&&\left\vert \left( \Omega \varphi \right) \left( t\right) -\left( \Omega \mu \right) \left( t\right) \right\vert  \notag \\ &\leq &\left\vert \frac{1}{\Gamma \left(
\alpha \right) }\int_{t_{0}}^{t}N^{\alpha}_{\psi}(t,u)\left( F\left(
u,\varphi \left( u\right) ,\varphi \left( u-a\right) \right) -F\left( u,\mu \left( u\right) ,\mu \left( u-a\right) \right) \right) du\right\vert  \notag \\
&\leq &\frac{1}{\Gamma \left( \alpha \right) }\int_{t_{0}}^{t}N^{\alpha}_{\psi}(t,u)\left( L_{1}\left\vert \varphi \left( u\right) -\mu \left( u\right) \right\vert -L_{2}\left\vert \varphi \left( u-a\right) -\mu \left(
u-a\right) \right\vert \right) du  \notag \\ &\leq &\frac{ML_{1}}{\Gamma \left( \alpha \right) }\int_{t_{0}}^{t}N^{\alpha}_{\psi}(t,u)\phi \left( u\right) du+\frac{ML_{2}}{\Gamma \left(
\alpha \right) }\int_{t_{0}}^{t}N^{\alpha}_{\psi}(t,u)\phi \left(
u\right) du  \notag \\
&\leq &\frac{M\left( L_{1}+L_{2}\right) }{\Gamma \left( \alpha \right) }\left\vert \int_{t_{0}}^{t}N^{\alpha}_{\psi}(t,u)\phi \left( u\right) du\right\vert  \notag \\
&\leq &MK\left( L_{1}+L_{2}\right) \phi \left( t\right) ,\text{ }t\in \left[t_{0},T\right]
\end{eqnarray*}
and
\begin{equation*}
\left\vert \left( \Omega \varphi \right) \left( t\right) -\left( \Omega \mu \right) \left( t\right) \right\vert =\Phi \left( t\right) -\Phi \left(
t\right) =0,\text{ }t\in \left[ t_{0}-a,t\right]
\end{equation*}
which implies that $d\left( \Omega \varphi -\Omega \mu \right) \leq K\left( L_{1}+L_{2}\right) d\left( \varphi ,\mu \right) .$ Since $0<K\left( L_{1}+L_{2}\right) <1$, then $\Omega $ is strictly contractive on $S$.
Let $\xi \in S$ arbitrary and $\underset{t\in I}{\min }\text{ }\phi \left( t\right) >0$. As $F\left( t,\xi \left( t\right) ,\xi \left(t-a\right) \right) $ and $\xi \left( t\right) $ are bounded on $I$, then exists a constant $0<M<\infty $ such that
\begin{eqnarray}\label{eq7}
&&\left\vert \left( \Omega \xi \right) \left( t\right) -\xi \left( t\right) \right\vert  \notag \\&=&\left\vert \Psi^{\gamma}(t,t_{0}) \Phi \left( t_{0}\right) +\frac{1}{\Gamma \left( \alpha \right) }\int_{t_{0}}^{t}N^{\alpha}_{\psi}(t,u)F\left( u,\xi \left( u\right) ,\xi \left( u-a\right)
\right) du-\xi \left( t\right) \right\vert  \notag \\
&\leq &M\varphi \left( t\right).
\end{eqnarray}
Thus, by means of Eq.(\ref{eq7}), it follows that $d\left( \Omega \xi ,\xi \right) <\infty $. By means of the Theorem \ref{teo1} (1), there exists a continuous function $
y_{0}:I\rightarrow \mathbb{R}$ such that $\Omega ^{n}\xi \rightarrow y_{0}$ in $\left( S,d\right) $ and $\Omega y_{0}=y_{0}$, then $y_{0}$ satisfies
\begin{equation*}
\left\{
\begin{array}{cll}
^{H}\mathbb{D}_{t_{0}+}^{\alpha ,\beta ;\psi }y_{0}\left( t\right) & = & F\left( t,y_{0}\left( t\right) ,y_{0}\left( t-a\right) \right) ,\text{ }t\in \left[ t_{0},T\right] \\
y_{0}\left( t\right) & = & \Phi \left( t\right),\text{ }t\in \left[ t_{0}-a,t_{0}\right].
\end{array}
\right.
\end{equation*}
Now consider for any $g\in S$, such that $g$ and $\xi$ are bounded on $I$, then exist a constant $0<M_{g}<\infty $ such that
\begin{equation*}
\left\vert \xi \left( t\right) -g\left( t\right) \right\vert \leq M_{g}\varphi \left( t\right)
\end{equation*}
for $t\in I$.
Thus, we can write $\forall g\in S$, $d\left( \xi ,g\right) <\infty $ with $S=\left\{ g\in S;d\left( \xi ,g\right) <\infty \right\}$. Furthermore, it is clear that
\begin{equation}\label{eq8}
-\phi \left( t\right)\leq \text{ } ^{H}\mathbb{D}_{t_{0}+}^{\alpha ,\beta ;\psi }y\left(t\right) -F\left( t,y\left( t\right) ,y\left( t-a\right) \right) \leq \phi
\left( t\right) ,\text{ }\forall t\in \left[ t_{0},T\right].
\end{equation}
Applying the fractional integral $I_{t_{0}}^{\alpha ;\psi }\left( \cdot\right) $ on both sides of Eq.(\ref{eq8}), we have
\begin{eqnarray*}
&&\left\vert y\left( t\right) -\Psi^{\gamma}(t,t_{0})\Phi
\left( t_{0}\right) -\frac{1}{\Gamma \left( \alpha \right) }\int_{t_{0}}^{t}N^{\alpha}_{\psi}(t,u) F\left( u,y\left( u\right) ,y\left( u-a\right) \right) du\right\vert  \notag \\
&\leq &\left\vert \frac{1}{\Gamma \left( \alpha \right) }\int_{t_{0}}^{t}N^{\alpha}_{\psi}(t,u)\phi \left( u\right) du\right\vert \leq K\phi\left( t\right) ,\text{ }t\in \left[ t_{0},T\right] .
\end{eqnarray*}
This form, by definition $\Omega $, finishes
\begin{equation*}
\left\vert y\left( t\right) -\left( \Omega y\right) \left( t\right)\right\vert \leq K\phi \left( t\right) ,\text{ }t\in I.
\end{equation*}
Consequently, it implies that $\ d\left( y,\Omega y\right) \leq K$. By means of Theorem \ref{teo1} (3) and the last estimative, we have
\begin{equation*}
d\left( y,y_{0}\right) \leq \frac{1}{1-K\left( L_{1}+L_{2}\right) }d\left( \Omega y,y\right) \leq \frac{K\phi \left( t\right) }{1-K\left(
L_{1}+L_{2}\right) },\text{ }\forall t\in I.
\end{equation*}
Thus, by Theorem \ref{teo1} (2), we conclude that there exists $y_{0}$, the unique continuous function with the property Eq.(\ref{eq3}).
\end{proof}

\begin{remark}
One of the advantages of working with Ulam-Hyers and Ulam-Hyers-Rassias
stabilities, or any other type of stability with the global fractional
differential operator so-called $\psi$-Hilfer, is that the results obtained in
this way, are also valid for their respective individual cases.
\end{remark}


\section{Ulam-Hyers stability}

In this section, we investigate the second main result of the paper, the
Ulam-Hyers stability, again making use of the Banach fixed point theorem.

\begin{teorema}
\label{teo3} Suppose that $F:I\times\mathbb{R}\times\mathbb{R}\rightarrow
\mathbb{R}$ is a continuous function with the following Lipschitz condition
\[
\left\vert F\left(  t,x,y\right)  -F\left(  t,z,w\right)  \right\vert \leq
L_{1}\left\vert x-z\right\vert +L_{2}\left\vert y-w\right\vert ,
\]
where $\left(  t,x,y\right)  ,\left(  t,z,w\right)  \in I\times\mathbb{R}%
\times\mathbb{R}$ and $0<\dfrac{\left(  \psi\left(  T\right)  \right)
^{\alpha}\left(  L_{1}+L_{2}\right)  }{\Gamma\left(  \alpha+1\right)  }<1$.

Let $\Phi\in C_{1-\gamma;\psi}\left[  t_{0}-a,t\right]  $ and $\varepsilon
\geq0.$ If a continuous function $y:I\rightarrow\mathbb{R}$ satisfies
\[
\left\{
\begin{array}
[c]{cll}%
\left\vert ^{H}\mathbb{D}_{t_{0}+}^{\alpha,\beta;\psi}y\left(  t\right)
-F\left(  t,y\left(  t\right)  ,y\left(  t-a\right)  \right)  \right\vert  &
< & \varepsilon,\text{ }t\in\left[  t_{0},T\right] \\
\left\vert y\left(  t\right)  -\Phi\left(  t\right)  \right\vert  & < &
\varepsilon,\text{ }t\in\left[  t_{0}-a,t_{0}\right]
\end{array}
\right.
\]
then there exists a unique continuous function $y_{0}:I\rightarrow\mathbb{R}$
such that
\begin{equation}
\label{eq9}\left\{
\begin{array}
[c]{cll}%
^{H}\mathbb{D}_{t_{0}+}^{\alpha,\beta;\psi}y_{0}\left(  t\right)  & = &
F\left(  t,y_{0}\left(  t\right)  ,y_{0}\left(  t-a\right)  \right)  ,\text{
}t\in\left[ t_{0},T\right] \\
y_{0}\left(  t\right)  & = & \Phi\left(  t\right)  ,\text{ }t\in\left[
t_{0}-a,t_{0}\right]
\end{array}
\right.
\end{equation}
and
\begin{equation}
\label{eq10}\left\vert y\left(  t\right)  -y_{0}\left(  t\right)  \right\vert
\leq\frac{\varepsilon\left(  \psi\left(  T\right)  \right)  ^{\alpha}}%
{\Gamma\left(  \alpha+1\right)  -\left(  \psi\left(  T\right)  \right)
^{\alpha}\left( L_{1}+L_{2}\right)  },\text{ }\forall t\in I.
\end{equation}

\end{teorema}

\begin{proof}
For the proof, we consider the following generalized metric over $S$, given by
\begin{equation*}
d_{1}\left( \varphi ,\mu \right) =\inf \left\{ M\in \left[ 0,\infty \right] :\left\vert \varphi \left( t\right) -\mu \left( t\right) \right\vert \leq M,
\text{ }\forall t\in I\right\}.
\end{equation*}
Note that, $\left( S,d_{1}\right) $ is a generalized complete metric space.
For any $\varphi ,\mu \in S$ and $M_{\varphi ,\mu }\in \left\{ M\in \left[
0,\infty \right] :\left\vert \varphi \left( t\right) -\mu \left( t\right) \right\vert \leq M,\text{ }\forall t\in I\right\}$, using Eq.(\ref{eq6}), we obtain
\begin{eqnarray*}
&&\left\vert \left( \Omega \varphi \right) \left( t\right) -\left( \Omega \mu \right) \left( t\right) \right\vert  \notag \\
&=&\left\vert
\begin{array}{c}
\dfrac{1}{\Gamma \left( \alpha \right) }\displaystyle\int_{t_{0}}^{t}N^{\alpha}_{\psi}(t,u) F\left( u,\phi \left( u\right) ,\phi \left( u-a\right) \right) du \\
-\dfrac{1}{\Gamma \left( \alpha \right) }\displaystyle\int_{t_{0}}^{t}N^{\alpha}_{\psi}(t,u)F\left( u,\mu \left( u\right) ,\mu \left( u-a\right) \right) du
\end{array}%
\right\vert  \notag \\
&\leq &\frac{1}{\Gamma \left( \alpha \right) }\int_{t_{0}}^{t}N^{\alpha}_{\psi}(t,u) \left( L_{1}\left\vert \phi \left( u\right) -\mu \left( u\right) \right\vert +L_{2}\left\vert \phi \left( u-a\right) -\mu \left( u-a\right)
\right\vert \right) du  \notag \\
&\leq &\frac{L_{1}M_{\varphi ,\mu }}{\Gamma \left( \alpha \right) }\int_{t_{0}}^{t}N^{\alpha}_{\psi}(t,u)du+\frac{L_{2}M_{\varphi ,\mu }}{\Gamma \left( \alpha \right) }\int_{t_{0}}^{t}N^{\alpha}_{\psi}(t,u)du \notag \\
&\leq &\frac{\left( L_{1}+L_{2}\right) M_{\varphi ,\mu }}{\Gamma \left(\alpha +1\right) }\left( \psi \left( T\right) \right) ^{\alpha }
\end{eqnarray*}
and
\begin{equation*}
\left\vert \left( \Omega \varphi \right) \left( t\right) -\left( \Omega \mu \right) \left( t\right) \right\vert =\Phi \left( t\right) -\Phi \left(
t\right) =0,\text{ }\forall t\in \left[ t_{0}-a,t_{0}\right]
\end{equation*}
which imply that $d_{1}\left( \Omega \varphi ,\Omega \mu \right) \leq \dfrac{\left( L_{1}+L_{2}\right) \left( \psi \left( T\right) \right) ^{\alpha }}{\Gamma \left( \alpha +1\right) }d\left( \varphi ,\mu \right)$.
Since $0<\dfrac{\left( \psi \left( T\right) \right) ^{\alpha }\left( L_{1}+L_{2}\right) }{\Gamma \left( \alpha +1\right) }<1,$ then $\Omega $ is
strictly contractive on $S$.
Now, for an arbitrary $\delta \in S$ and using the fact that $F\left( t,\delta \left( t\right) ,\delta \left( t-a\right) \right) $ and $\delta \left( t\right)$, are boundedness on $S$, we can show $d_{1}\left( \Omega \delta ,\delta \right) <\infty $. Hence, from Theorem \ref{teo1} (1), there exists a continuous function $y_{0}:I\rightarrow
\mathbb{R}$ such that $\Omega ^{n}\xi \rightarrow y_{0}$ in $\left( S,d_{1}\right) $ and $\Omega y_{0}=y_{0},$ then $y_{0}$ satisfies
\begin{equation*}
\left\{
\begin{array}{cll}
^{H}\mathbb{D}_{t_{0}+}^{\alpha ,\beta ;\psi }y_{0}\left( t\right) & = & F\left(t,y_{0}\left( t\right) ,y_{0}\left( t-a\right) \right) ,\text{ }t\in \left[ t_{0},T\right] \\
y_{0}\left( t\right) & = & \Phi \left( t\right) ,\text{ }t\in \left[t_{0}-a,t_{0}\right].
\end{array}
\right.
\end{equation*}
Using the fact that $g$ and $\delta$ are bounded on $I$, then $d_{1}\left( \delta ,g\right) <\infty $, $\forall g\in S$, with $S=\left\{ d_{1}\left( \delta,y\right) <\infty \right\}$.
Then, using Theorem \ref{teo1} (2), $y_{0}$ is the unique continuous function with the property Eq.(\ref{eq9}).
Note that,
\begin{equation}\label{eq11}
-\varepsilon \leq \text{ } ^{H}\mathbb{D}_{t_{0}+}^{\alpha ,\beta ;\psi }y\left( t\right)-F\left( t,y\left( t\right) ,y\left( t-a\right) \right) \leq \varepsilon
\end{equation}
for all $t\in \left[ t_{0},T\right]$.
Applying the fractional integral $I_{t_{0}}^{\alpha ;\psi }\left( \cdot \right) $, on both sides of Eq.(\ref{eq11}), we get
\begin{eqnarray}\label{eq12}
&&\left\vert y\left( t\right) -\Psi^{\gamma}(t,t_{0})-\frac{1}{\Gamma \left( \alpha \right) }\int_{t_{0}}^{t}N^{\alpha}_{\psi}(t,u)F\left( u,y\left( u\right) ,y\left( u-a\right) \right) du\right\vert \notag \\
&\leq &\varepsilon I_{t_{0}}^{\alpha ;\psi }\left( 1\right) \leq \frac{\varepsilon \left( \psi \left( T\right) \right) ^{\alpha }}{\Gamma \left(
\alpha +1\right) }
\end{eqnarray}
for each $t\in I$.
By means of Theorem \ref{teo1} (3) and Eq.(\ref{eq12}), we get
\begin{eqnarray*}
d_{1}\left( y,y_{0}\right)  &\leq &\frac{\varepsilon \left( \psi \left( T\right) \right) ^{\alpha }}{\Gamma \left( \alpha +1\right) \left( 1-\frac{%
\left( \psi \left( T\right) \right) ^{\alpha }\left( L_{1}+L_{2}\right) }{\Gamma \left( \alpha +1\right) }\right) }  \notag \\
&=&\frac{\varepsilon \left( \psi \left( T\right) \right) ^{\alpha }}{\Gamma \left( \alpha +1\right) -\left( \psi \left( T\right) \right) ^{\alpha
}\left( L_{1}+L_{2}\right) },
\end{eqnarray*}
which we conclude the proof.
\end{proof}

\begin{corolario}
\label{cor1} Suppose the conditions of the \textrm{Theorem \ref{teo2}}. If a
continuous function $y:I\rightarrow\mathbb{R}$ satisfies
\begin{equation}
\left\{
\begin{array}
[c]{cll}%
\left\vert \mathcal{D}_{t_{0}+}^{\alpha}y\left(  t\right)  -F\left(
t,y\left(  t\right) ,y\left(  t-a\right)  \right)  \right\vert  & < &
\varepsilon,\text{ }t\in\left[  t_{0},T\right] \\
\left\vert y\left(  t\right)  -\Phi\left(  t\right)  \right\vert  & < &
\varepsilon,\text{ }t\in\left[  t_{0}-a,t_{0}\right]
\end{array}
\right.
\end{equation}
then there exists a unique continuous function $y_{0}:I\rightarrow\mathbb{R}$
such that
\begin{equation}
\left\{
\begin{array}
[c]{cll}%
\mathcal{D}_{t_{0}+}^{\alpha}y_{0}\left(  t\right)  & = & F\left(
t,y_{0}\left(  t\right)  ,y_{0}\left(  t-a\right)  \right)  ,\text{ }%
t\in\left[  t_{0},T\right] \\
y_{0}\left(  t\right)  & = & \Phi\left(  t\right)  ,\text{ }t\in\left[
t_{0}-a,t_{0}\right]
\end{array}
\right.
\end{equation}
and
\begin{equation}
\left\vert y\left(  t\right)  -y_{0}\left(  t\right)  \right\vert \leq
\frac{\varepsilon\left(  \ln T\right)  ^{\alpha}}{\Gamma\left(  \alpha
+1\right)  -\left(  \ln T\right)  ^{\alpha}\left(  L_{1}+L_{2}\right)
},\text{ }\forall t\in I,
\end{equation}
where $\mathcal{D}_{t_{0}+}(\cdot)$ is the Hadamard fractional derivative.
\end{corolario}

\begin{proof}
The proof is a direct consequence of the Theorem \ref{teo2}.
\end{proof}

\begin{corolario}
\label{cor2} Suppose the conditions of the \textrm{Theorem \ref{teo2}}. If a
continuous function $y:I\rightarrow\mathbb{R}$ satisfies
\begin{equation}
\left\{
\begin{array}
[c]{cll}%
\left\vert y^{\prime}\left(  t\right)  -F\left(  t,y\left(  t\right) ,y\left(
t-a\right)  \right)  \right\vert  & < & \varepsilon,\text{ }t\in\left[
t_{0},T\right] \\
\left\vert y\left(  t\right)  -\Phi\left(  t\right)  \right\vert  & < &
\varepsilon,\text{ }t\in\left[  t_{0}-a,t_{0}\right]
\end{array}
\right.
\end{equation}
then there exists a unique continuous function $y_{0}:I\rightarrow\mathbb{R}$
such that
\begin{equation}
\left\{
\begin{array}
[c]{cll}%
y^{\prime}_{0}\left(  t\right)  & = & F\left(  t,y_{0}\left(  t\right)
,y_{0}\left(  t-a\right)  \right)  ,\text{ }t\in\left[  t_{0},T\right] \\
y_{0}\left(  t\right)  & = & \Phi\left(  t\right)  ,\text{ }t\in\left[
t_{0}-a,t_{0}\right]
\end{array}
\right.
\end{equation}
and
\begin{equation}
\left\vert y\left(  t\right)  -y_{0}\left(  t\right)  \right\vert \leq
\frac{\varepsilon T}{1- T\left(  L_{1}+L_{2}\right)  },\text{ }\forall t\in I.
\end{equation}

\end{corolario}

\begin{proof}
The proof is a direct consequence of the Theorem \ref{teo2}.
\end{proof}

\begin{remark}
The following fractional differential equation
\begin{equation}
\label{eq49}^{H}\mathbb{D}_{t_{0}+}^{\alpha,\beta;\psi}y\left(  t\right) =
F\left(  t,y\left(  t\right)  \right)
\end{equation}
is a special case of \textrm{Eq.(\ref{eq1})}. Consequently, the results
proposed here are also valid for \textrm{Eq.(\ref{eq49})}.

Applying the limit $\alpha\rightarrow1$ on both sides of the
\textrm{Eq.(\ref{eq49})}, we obtain the following differential equation of the
first order \textrm{\cite{jung}}
\[
y^{\prime}\left(  t\right) = F\left(  t,y\left(  t\right)  \right)  ,
\]
which, in turn, the results proposed here, are also valid.
\end{remark}


\section{Concluding Remarks}

The study of Ulam-Hyers-type stability of solutions of the fractional functional differential equations has been the object of much study and investigated by many researchers \cite{exis1,exis2,exis3,cont1,cont2,cont3,ist1,ist2,ist3}. Although it is yet a field of mathematics that is in expansion, over the years countless works have been published and others are yet to come. In this sense, the paper presented a discuss on the Ulam-Hyers and Ulam-Hyers-Rassias stabilities of the fractional functional differential equation Eq.(\ref{eq1}) through the Banach fixed point theorem, which contributes to the growth of this area.

From this contribution, the natural question that arises will be whether by means of the $\psi$-Hilfer fractional derivative it is also possible to obtain the stabilities investigated here in the function space $L_{p,\alpha
}(I,\mathbb{R})$? And using another fixed point theorem? Another possibility of study is to investigate other types of stabilities such as $\delta $-Ulam-Hyers-Rassias, semi-Ulam-Hyers-Rassias and Mittag-Leffler-Ulam using
the same fractional differentiable operator \cite{sousa,sousa1,mittag}. Studies in this direction are being prepared and will be published in the near future.

\bibliography{ref}
\bibliographystyle{plain}

\end{document}